\documentclass[english,11pt]{amsart}

\usepackage[english]{babel}
\usepackage[T1]{fontenc}
\usepackage{lmodern}

\usepackage{MnSymbol}
\usepackage[centering]{geometry}

\theoremstyle{plain}
\newtheorem{theo}{Theorem}
\newtheorem{lemm}[theo]{Lemma}
\newtheorem{prop}[theo]{Proposition}
\newtheorem{coro}[theo]{Corollary}

\theoremstyle{definition}

\theoremstyle{remark}
\newtheorem{rema}[theo]{Remark}

\newtheorem{theoA}{Theorem}

\usepackage{microtype}
\usepackage{graphicx}
\usepackage{bm}
\usepackage{mathtools}
\usepackage[dvipsnames]{xcolor}
\definecolor{FlatRed}{RGB}{231,76,60}
\definecolor{FlatGreen}{RGB}{46,204,113}
\definecolor{FlatBlue}{RGB}{52,152,219}
\definecolor{FlatYellow}{RGB}{241,196,15}
\colorlet{FlatViolet}{FlatRed!50!FlatBlue}
\colorlet{FlatBrown}{FlatRed!50!FlatGreen}
\colorlet{FlatOrange}{FlatRed!50!FlatYellow}
\colorlet{FlatCyan}{FlatGreen!50!FlatBlue}

\usepackage{enumitem}
\usepackage{hyperref}
\hypersetup{
    colorlinks,
    linkcolor={FlatRed},
    citecolor={FlatGreen},
    urlcolor={FlatBlue}
}

\title[General convex $H^{2|2}$ monotonicity theorem]{Supersymmetric integration by parts and multivariate convex monotonicity for the $H^{2|2}$ model}

\author{Yichao Huang}
\address{Beijing Institute of Technology, School of Mathematics and Statistics, Beijing, China}
\email{yichao.huang@outlook.com}

\author{Xiaolin Zeng}
\address{Institut de Recherche Mathématique Avancée, Université de Strasbourg, Strasbourg, France}
\email{zeng@math.unistra.fr}

\begin{document}

\begin{abstract}
  We use supersymmetric localization and integration by parts to derive variational and convex correlation inequalities in statistical physics. As a primary application, we give an alternative proof of the monotonicity theorem for the $H^{2|2}$ supersymmetric hyperbolic sigma model. This extends a result of Poudevigne to a stronger multivariate convex comparison inequality, without relying on probabilistic couplings.
\end{abstract}


\maketitle

\section{Introduction}
We extend the method of supersymmetric localization to establish continuous correlation inequalities in probability and statistical physics. As a pedagogical warm-up, we first revisit classical Gaussian comparison inequalities, then we apply this method to the $H^{2|2}$ supersymmetric hyperbolic sigma model introduced by Zirnbauer~\cite{MR1134935}. In particular, we revisit the monotonicity theorem for the $H^{2|2}$ model, originally established by Poudevigne~\cite{MR4721024} via discrete probabilistic couplings. Using the new supersymmetric localization approach, we obtain a stronger multivariate convex comparison inequality with a Kahane--Slepian type condition for the $H^{2|2}$ model, see Theorem~\ref{th:better_intro}.

Historically, localization formulas in supersymmetry have taken various forms in both the physics and mathematics communities~\cite{MR958805,MR1185834,MR2045303,MR2928219,MR674406,MR685019,MR721448}. These formulas have broad applications in mathematical physics; in supersymmetric quantum field theories, for example, supersymmetric reductions via localization formulas yield exact solutions~\cite{MR3718870}. Another important class of applications for supersymmetric localization involves Morse-type dimension-comparison inequalities in geometry and topology~\cite{MR683171}.

We will introduce the precise supersymmetric definition of the $H^{2|2}$ model later in the paper: we first recall the $H^{2|2}$ monotonicity theorem in Poudevigne's formulation~\cite{MR4721024} as a property of real-valued integrals, without using the supersymmetric language.
\begin{theoA}[Monotonicity theorem for the $H^{2|2}$ model]\label{th:main}
Fix an integer $N\geq 1$, consider the vertex set $V=\{1,\dots,N,\delta\}$ with a root vertex $\delta$, unoriented edges $(ij)=(ji)$ for $i,j\in V$ and $i\neq j$. Assign symmetric positive edge weights $W_{ij}=W_{ji}>0$ for $i,j\in V$ and $i\neq j$. Define
\begin{equation}\label{eq:det_matrixtree}
    \mathcal{D}_W(t_1,\dots,t_N)=\sum_{T\in\mathcal{T}}\prod_{(ij)\in T}W_{ij}e^{t_i+t_j}
\end{equation}
for $t=(t_1,\dots,t_N)$ and $t_{\delta}=0$, where the sum is over the spanning trees $\mathcal{T}$ of the complete graph on $\{1,\dots,N,\delta\}$. Consider the effective $H^{2|2}$ action functional
\begin{equation*}
    S^W_{\text{eff}}(t)=\frac{1}{2}\sum_{i,j\in V}W_{ij}(\cosh(t_i-t_j)-1)-\frac{1}{2}\left(\ln \mathcal{D}_{W}(t_1,\dots,t_N)-2\sum_{i=1}^{N}t_i\right).
\end{equation*}
Then for fixed $p_1,\dots,p_N\geq 0$, the integral
\begin{equation*}
    K=(2\pi)^{-\frac{N}{2}}\int_{\mathbb{R}^{N}}f\left(\sum_{i=1}^{N}p_ie^{t_i}\right)e^{-S^W_{\text{eff}}(t)}dt_1\cdots dt_{N}
\end{equation*}
is non-increasing in each of the $(W_{ij})_{i,j\in\{1,\dots,N,\delta\}}$ if $f:\mathbb{R}\to\mathbb{R}$ is convex.
\end{theoA}


An analogue of this monotonicity theorem is conjectured for the $H^{2|4}$ model, see the discussion after~\cite[Equation~(5.5)]{MR4680390}. To establish monotonicity for the $H^{2|2}$ model, Poudevigne~\cite{MR4721024} developed a probabilistic method. However, this method relies heavily on the specific structure of the $H^{2|2}$ setting; more specifically, three properties limit the broader application of this probabilistic approach. First, the proof uses a graph reduction property, which implies that studying a two-point graph suffices for the $H^{2|2}$ model (see Remark~\ref{rema:ReductionProperty} for a short alternative proof assuming this property). Second, it relies on the triviality of the $H^{2|2}$ partition function, which follows directly from the supersymmetric localization formula~\cite[Proposition~2]{MR2728731}. Third, it utilizes a discrete probabilistic coupling lemma that requires the first two properties, thereby confining the technique to the $H^{2|2}$ setting.

This paper provides an alternative, more general approach to the monotonicity theorem. Poudevigne's approach is motivated by the exact probabilistic representation of the $H^{2|2}$ model, building on the surprising discovery of~\cite{MR3420510} in connection to the Vertex Reinforced Jump Process (VRJP) and Edge Reinforced Random Walk (ERRW), for which he derived the uniqueness of phase transition using the $H^{2|2}$ monotonicity theorem. Our study is based on the original supersymmetric representation of the $H^{2|2}$ model: a continuous supersymmetric variational calculation entirely replaces the discrete probabilistic coupling lemma, and our proof avoids any specific $H^{2|2}$ graph reduction property.

With our new method, we generalize Theorem~\ref{th:main} to a stronger form by enlarging the class of convex observables. More precisely,
\begin{theo}[Generalized form of the $H^{2|2}$ monotonicity theorem]\label{th:better_intro}
Suppose $F: \mathbb{R}^N \to \mathbb{R}$ is smooth and jointly convex (i.e., it has a positive semi-definite Hessian). Then the integral
\begin{equation*}
    J=\int F(e^{t_1},\dots,e^{t_N})d\nu_\delta(\bm{t})
\end{equation*}
is non-increasing in each of the edge weights $W_{ij}$ with $i,j\in\{1,\dots,N,\delta\}$. 
\end{theo}
Therefore, we are able to extend the $H^{2|2}$ monotonicity to more general observables $F(e^{t_1},\dots,e^{t_N})$ with jointly convex functions $F:\mathbb{R}_{>0}^{N}\to\mathbb{R}$. In particular, we do not need to assume $p_1,\dots,p_N>0$ in the linear combination. Moreover, with this stronger form of the $H^{2|2}$ monotonicity theorem, one obtains similarly to the classical Slepian's lemma for Gaussian processes a new class of useful comparison inequalities for the $H^{2|2}$ model, such as
\begin{equation*}
    \mathbb{E}_{\bm{W'}}\left[\max_{1\leq i\leq N}e^{t_i}\right]\leq \mathbb{E}_{\bm{W}}\left[\max_{1\leq i\leq N}e^{t_i}\right]
\end{equation*}
if $\bm{W}\leq\bm{W'}$ by standard convex approximation. This type of consequences is not covered by the weaker formulation of Theorem~\ref{th:main}.

\subsection{Overview of our method}
The core of our method is the supersymmetric localization formula. Readers already familiar with this language may briefly skim through Sections~\ref{sec:sign-conv} and~\ref{sec:h22-1p-case} for our conventions and notation, then proceed directly to Section~\ref{subse:general_h_2_2} for a short, self-contained proof of the more general $H^{2|2}$ monotonicity theorem.

Originally developed for supersymmetric models, supersymmetric localization provides a powerful tool for exact calculations that extends far beyond intrinsically supersymmetric theories, provided that a suitable supersymmetric structure can be introduced. Many calculations in a supersymmetric theory can be exactly reduced to local contributions or semi-classical approximations by localizing onto a lower-dimensional critical manifold.

However, most applications of the supersymmetric localization formula concern exact equalities, and relatively few (to our knowledge, none systematically) deal with inequalities, especially those involving continuous parameters. Our starting point is the striking similarity between Theorem~\ref{th:main} and Kahane's convexity inequality (see~\cite[Section~3]{Huang:2025aa} for the more general multivariate jointly convex form, similar to the upgrade from Theorem~\ref{th:main} to Theorem~\ref{th:better} in this paper) or Slepian's inequality for Gaussian processes. The latter classical comparison inequalities can be proved by showing a monotonicity property under a rotation of the Euclidean coordinate system. Since the supersymmetric generator $Q$ squares to a rotation generator by Cartan's magic formula, it is natural to adapt the Gaussian interpolation method and to use $Q$ to perform (super-)integration by parts to establish convexity or comparison inequalities in supersymmetric statistical physics models. Although a generic supersymmetric integral lacks a deterministic sign, utilizing horospherical coordinates resolves this issue. Switching between coordinates and performing supersymmetric integration by parts in the correct setting results in a deterministic sign for the derivative even after integrating over the Grassmann variables.

Although many calculations in this paper are written using the language of supersymmetric integration, one could in principle translate them into probabilistic real-variable calculations by writing out explicitly all the Berezin determinants. However, the supersymmetric formulation makes the cancellations underlying the argument transparent. A geometric interpretation of Royen's proof of the Gaussian correlation inequality~\cite{Royen2014,Lata_a_2017} is obtained using the method of this paper in~\cite{Huang:2026aa}.

\subsection*{Acknowledgement}
Y.H. is partially supported by the National Key R\&D Program of China No. 2022YFA1006300 and NSFC-12301164, NSFC-12671177. X.Z. acknowledges the support of the Institut de recherche en mathématiques, interactions \& applications: IRMIA++.

\section{A short proof of the \texorpdfstring{$H^{2|2}$}{H2|2} monotonicity theorem}\label{sec:a_short_proof_of_the_h_2_2}

\subsection{Sign convention}\label{sec:sign-conv}
Since we deal with inequalities in this paper, we must first establish our sign conventions for the Berezin supersymmetric calculus \cite{berezin2013introduction}.

We denote bosonic variables by Latin letters $x, y, z, t, s, \dots$, and fermionic (or Grassmann) variables by Greek letters $\xi, \eta, \psi, \bar{\psi}, \dots$. We indicate vectors with boldface letters, such as $\bm{z}$ or $\bm{\xi}$. Fermionic variables anti-commute, meaning $\xi\eta = -\eta\xi$, in particular $\xi^2=0$.

The left fermionic derivative acts as an odd derivation. If $F$ does not depend on $\xi$, the derivative satisfies
\begin{equation*}
    \frac{\partial}{\partial\xi}(\xi F) = F, \quad \frac{\partial}{\partial\xi}F = 0.
\end{equation*}
Analogous rules apply to all other Grassmann variables. The fermionic integral formally matches the fermionic derivative. For any differentiable function $F$, we have
\begin{equation*}
    \int \partial_\xi F(\xi) = \frac{\partial}{\partial\xi}F(\xi).
\end{equation*}

An important consequence is the fermionic Gaussian integral: for any $N\times N$ matrix $\Sigma$,
\begin{equation}\label{eq:Fermionic_GaussianIntegral}
    \int \prod_{i=1}^{N}\partial_{\xi_i}\partial_{\eta_i}e^{-\bm{\xi}^{T}\Sigma\bm{\eta}}=\det(\Sigma),
\end{equation}
where $\bm{\xi}=(\xi_1,\dots,\xi_N)$ and similarly for $\bm{\eta}$.

A distinguished supersymmetric operator $Q$ plays a crucial role in the sequel:
\begin{equation}\label{eq:Q_Operator}
    Q=\xi\frac{\partial}{\partial x}+\eta\frac{\partial}{\partial y}+x\frac{\partial}{\partial\eta}-y\frac{\partial}{\partial\xi}.
\end{equation}
This operator interchanges the bosonic and fermionic variables. It is easily seen that it annihilates the distinguished even variable
\begin{equation}\label{eq:H_Variable}
    H=x^2+y^2+2\xi\eta
\end{equation}
in the sense that $Q(H)=0$, and that $H$ is also $Q$-exact since
\begin{equation}\label{eq:Lambda_Variable}
    H=Q(\lambda),\quad \lambda=x\eta-y\xi.
\end{equation}

Let $f$ and $g$ be homogeneous elements of the Grassmann superalgebra. The Grassmann parity $|f|\in\{0,1\}$ indicates whether $f$ is bosonic ($|f|=0$, an even element) or fermionic ($|f|=1$, an odd element). The operator $Q$ satisfies the (super-)Leibniz rule
\begin{equation*}
    Q(fg)=(Qf)g+(-1)^{|f|}fQ(g).
\end{equation*}

When an even function $F$ takes fermionic arguments, we expand it using a Taylor series. Since fermionic variables are nilpotent, any Taylor series in these variables truncates, meaning $F$ is simply a polynomial in the Grassmann variables.

We also use the supersymmetric localization formula~\cite[Lemma~15]{MR2728731}. For any smooth function $f=f(x,y,\xi,\eta)$  with sufficient decay so that the bosonic boundary integrals vanish, we have
\begin{equation}\label{eq:Q_Localization_Formula}
    \int d\mu\, Q(f)=0,
\end{equation}
where the $Q$-invariant (flat) Berezin form is $d\mu=\frac{1}{2\pi}dxdy\partial_\xi\partial_\eta$. This invariance directly yields a $Q$-integration by parts formula. If $f$ and $g$ are smooth functions of $(x,y,\xi,\eta)$ with sufficient decay so that the bosonic boundary integrals vanish, then
\begin{equation}\label{eq:Q_Integration_by_Parts}
    \int d\mu\,Q(f)g=\int d\mu\,Q(fg)-(-1)^{|f|} \int d\mu\,fQ(g)=-(-1)^{|f|}\int d\mu\, fQ(g).
\end{equation}
The resulting sign depends strictly on the Grassmann parity $|f|$.

\subsection{Slepian-type inequality with supersymmetry}
To demonstrate the method in a familiar setting, we first prove the standard convexity inequality for a one-dimensional Gaussian variable. Let $\mathcal{N}$ be a centered normal random variable with variance $W^{-1}>0$.
\begin{prop}[The simplest Gaussian convex inequality]\label{prop:Slepian_2d}
The integral
\begin{equation*}
    I=\int_{\mathbb{R}}f(x)e^{-\frac{W}{2}x^2}\frac{\sqrt{W}}{\sqrt{2\pi}}dx=\mathbb{E}[f(\mathcal{N})]
\end{equation*}
is non-increasing in $W>0$ if $f:\mathbb{R}\to\mathbb{R}$ is convex.
\end{prop}
Jensen's inequality yields a one-line proof, but the goal here is to illustrate the implementation of the supersymmetric localization formula~\eqref{eq:Q_Localization_Formula}.

\begin{proof}
Recall the flat Berezin measure $d\mu=\frac{1}{2\pi}dxdy\partial_\xi\partial_\eta$ and the even variable $H$ defined in~\eqref{eq:H_Variable}. We claim that
\begin{equation*}
    I=\int d\mu f(x)e^{-\frac{W}{2}(x^2+y^2+2\xi\eta)}=\int d\mu f(x)e^{-\frac{W}{2}H}.
\end{equation*}

To see this equality, integrating out the Grassmann variables using \eqref{eq:Fermionic_GaussianIntegral} yields a factor of $W$. Integrating out the real variable $y$ yields $\int e^{-\frac{W}{2}y^2} dy = \sqrt{2\pi/W}$. Combining these with the flat measure $\frac{1}{2\pi}dxdy$, the effective measure for $x$ is precisely $\frac{W}{2\pi}\sqrt{\frac{2\pi}{W}}dx=\sqrt{\frac{W}{2\pi}}dx$.

Taking the derivative, we get
\begin{equation*}
    \frac{\partial}{\partial W}I=-\frac{1}{2}\int d\mu f(x)He^{-\frac{W}{2}H}=-\frac{1}{2}\int d\mu f(x)Q(\lambda e^{-\frac{W}{2}H})
\end{equation*}
using that $Q(\lambda)=H$ with $\lambda=x\eta-y\xi$ as in~\eqref{eq:Lambda_Variable} and $Q(H)=0$. Integrating by parts with respect to $Q$ using~\eqref{eq:Q_Integration_by_Parts}, we get
\begin{equation*}
    \frac{1}{2}\int d\mu Q(f(x))\lambda e^{-\frac{W}{2}H}=\frac{1}{2}\int d\mu f'(x)\xi\lambda e^{-\frac{W}{2}H}=\frac{1}{2}\int d\mu f'(x)(x\xi\eta)e^{-\frac{W}{2}H}.
\end{equation*}
Define the fermionic variable $\nu=-xy\xi-y^2\eta$ such that $Q(\nu)=\xi\lambda=x\xi\eta$. Then $Q$-integrating by parts again yields
\begin{equation*}
    \frac{1}{2}\int d\mu f'(x)Q(\nu)e^{-\frac{W}{2}H}=-\frac{1}{2}\int d\mu Q(f'(x))\nu e^{-\frac{W}{2}H}=-\frac{1}{2}\int d\mu f''(x)\xi\nu e^{-\frac{W}{2}H}.
\end{equation*}
But $\xi\nu=-y^2\xi\eta$, and integrating out the Grassmann variables $\xi,\eta$ yields
\begin{equation*}
    \frac{\partial}{\partial W}I=-\frac{1}{2}\int_{\mathbb{R}^2}\frac{dxdy}{2\pi} f''(x)y^2 e^{-\frac{W}{2}(x^2+y^2)}\leq 0,
\end{equation*}
since the last expression is an integration with real variables, $f$ is convex, and $y^2\geq 0$.
\end{proof}

One checks after integrating out $y$ that the derivative is exactly $-\frac{1}{2W^2}\mathbb{E}[f''(\mathcal{N})]$, which is expected via the classical Stein's lemma.\footnote{The statement and proof above are easily generalized to any higher-dimensional Gaussian vectors: one obtains in this manner the so-called dual-type Slepian Gaussian comparison inequalities (dual in the sense that we vary the inverse of the covariance matrix). We omit the general computations.}

\subsection{Supersymmetric proof of the \texorpdfstring{$H^{2|2}$}{H2|2} monotonicity theorem}\label{sec:h22-1p-case}
As a further illustration, we give a short supersymmetric proof of the monotonicity theorem for the $H^{2|2}$ model. The original proof of this result is due to Poudevigne~\cite[Theorem~6]{MR4721024}, who worked directly with real-valued integrals and relied on probabilistic tools such as discrete coupling techniques and Jensen's inequality. Our alternative proof in this section concisely replaces the technical coupling construction of Poudevigne~\cite[Section~3]{MR4721024}.

\subsubsection{Definition and statement}
We recall the definition of the $H^{2|2}$ supersymmetric hyperbolic sigma model introduced by Zirnbauer~\cite{MR1134935} following the presentations of~\cite{MR2728731,MR3420510,MR3904155,MR4218682}. We define the special bosonic variable $z_i=\sqrt{1+x_i^2+y_i^2+2\xi_i\eta_i}$ (defined via its terminating Taylor series) to form the $\mathbb{R}^{3|2}$ spin $v_i=(x_i,y_i,z_i,\xi_i,\eta_i)$ for $i\in V=\{1,\dots,N,\delta\}$, with inner product
\begin{equation*}
    v_{i}\cdot v_{j}=x_ix_j+y_iy_j-z_iz_j+\xi_i\eta_j+\xi_j\eta_i,
\end{equation*}
so that the hyperbolic constraint $v_i\cdot v_i=-1$ is satisfied for all $i\in V$. We also impose the boundary condition $v_{\delta}=(0,0,1,0,0)$, the latter being the base point of the $H^{2|2}$ manifold. The distinguished supersymmetric operator $Q$ is defined by
\begin{equation*}
    Q=\sum_{i=1}^{N}Q_i=\sum_{i=1}^{N}\left(\xi_i\frac{\partial}{\partial x_i}+\eta_i\frac{\partial}{\partial y_i}+x_i\frac{\partial}{\partial\eta_i}-y_i\frac{\partial}{\partial\xi_i}\right).
\end{equation*}

Given symmetric positive edge weights $W_{ij}=W_{ji}>0$ for $i,j\in V$, the $H^{2|2}$ action is
\begin{equation}\label{eq:H22_Action}
    S^{W}=\frac{1}{4}\sum_{i,j\in\{1,\dots,N,\delta\}}W_{ij}(v_i-v_j)^2=-\frac{1}{2}\sum_{i,j\in\{1,\dots,N\}}W_{ij}(v_i\cdot v_j+1)+\sum_{i\in\{1,\dots,N\}}W_{i\delta}(z_i-1).
\end{equation}
As $Q$ annihilates all inner products $v_i\cdot v_j$ and $z_i$, this action is $Q$-closed. Finally, the (hyperbolic) $Q$-invariant Berezin volume form for the $H^{2|2}$ model is
\begin{equation*}
    D\mu=\left(\prod_{l=1}^{N}\frac{1}{z_l}d\mu_l\right)=\left(\prod_{l=1}^{N}\frac{1}{2\pi z_l}dx_ldy_l\partial_{\xi_l}\partial_{\eta_l}\right).
\end{equation*}

The monotonicity theorem for the $H^{2|2}$ model with the formulation of Theorem~\ref{th:main} has the following equivalent description in the above Euclidean $(x,y,z,\xi,\eta)$ coordinates.
\begin{theoA}[Monotonicity theorem for $H^{2|2}$~\cite{MR4721024}]\label{theo:H22_Monotonicity}
Let $p_1,\dots,p_N\geq 0$ be fixed non-negative parameters. For any smooth convex function $f:\mathbb{R}\to\mathbb{R}$, the following integral with the action defined in~\eqref{eq:H22_Action},
\begin{equation*}
    K=\int D\mu f\left(\sum_{k=1}^{N}p_k(x_k+z_k)\right)e^{-S^{W}}
\end{equation*}
is non-increasing in the parameters $W_{ij}$ for any $i,j\in\{1,\dots,N,\delta\}$.
\end{theoA}

To make this statement rigorously well-defined, we recall the (super-)horospherical coordinates~\cite[Section~2.2]{MR2728731}, defined by the change of variables
\begin{equation}\label{eq:HorosphericalCoordinates}
\begin{split}
    &x_i=\sinh(t_i)-\left(\frac{1}{2}s_i^2+\bar{\psi}_i\psi_i\right)e^{t_i},\quad y_i=s_ie^{t_i},\quad z_i=\cosh(t_i)+\left(\frac{1}{2}s_i^2+\bar{\psi}_i\psi_i\right)e^{t_i},\\
    &\xi_i=\bar{\psi}_i e^{t_i},\quad \eta_i=\psi_i e^{t_i}.
\end{split}
\end{equation}
Under this change of variables, the hyperbolic measure $D\mu$ becomes the flat measure $d\mu(t,s,\bar{\psi},\psi)$ weighted by \(\prod_{\ell=1}^{N} e^{-t_{\ell}}\). For any smooth bounded function $F(\bm{v})=F(v_1,\dots,v_N)$, we have the equality
\begin{equation*}
    \int D\mu F(\bm{v})=\int d\mu(t,s,\bar{\psi},\psi) \left(\prod_{l=1}^{N}e^{-t_l}\right)F(\bm{v}),
\end{equation*}
where $d\mu(t,s,\bar{\psi},\psi)$ is the flat Berezin form in the new coordinates.

In these coordinates, conditional on the $t$-components, the action functional decouples and becomes quadratic in the variables $s$ and in $\bar{\psi},\psi$~\cite[Section~2.2]{MR2728731}:
\begin{equation}\label{eq:GaussianAction_Horo}
\begin{split}
    S^{W} &= \frac{1}{2} \sum_{i,j\in \{1,\ldots,N\}}W_{ij}\left(\cosh(t_i-t_j)-1 + \left(\frac{1}{2}(s_i-s_j)^2 +(\bar{\psi}_{i}-\bar{\psi}_{j})(\psi_{i}-\psi_{j})\right) e^{t_i+t_j}\right)\\
    & + \sum_{i\in\{1,\ldots,N\}}W_{i\delta} \left( \cosh t_i-1 + \left( \frac{1}{2}s_i^2 + \bar{\psi}_{i}\psi_{i} \right)e^{t_i}\right).
\end{split}
\end{equation}

The corresponding covariance matrix is the inverse of the $N\times N$ graph Laplacian $(\Delta^{W}_t)$, which is an $M$-matrix defined by
\begin{equation}\label{eq:M-LaplaceMatrix}
(\Delta^{W}_t)_{ij}=
\begin{cases}
    -W_{ij}e^{t_i+t_j}&\quad\text{if}\quad i\neq j,\\
    W_{i\delta}e^{t_i}+\sum_{j\neq i}W_{ij}e^{t_i+t_j}&\quad\text{if}\quad i=j.
\end{cases}
\end{equation}
Therefore, Gaussian integration over the bosonic variables $s$ produces a factor of $(\det \Delta_t^W)^{-1/2}$, while the fermionic Gaussian integration over $\psi, \bar{\psi}$ produces a factor of $\det \Delta_t^W$. The net result is $\sqrt{\det \Delta_t^W}$, which, by the Matrix-Tree Theorem, produces the $\sqrt{\mathcal{D}_W(\bm{t})}$ term. Together with the purely bosonic $\cosh$ terms from $S^W$ and the measure factor $\prod_{\ell=1}^N e^{-t_\ell}$, this exactly recovers the real-variable effective action $S^W_{\text{eff}}(t)$. Thus, integrating out all variables except $\bm{t}$ leaves the real-valued integral
\begin{equation}\label{eq:mixing_measure}
    K=\int_{\mathbb{R}^{N}}f\left(\sum_{k=1}^{N}p_ke^{t_k}\right)d\nu_{\delta}(\bm{t}),
\end{equation}
where $d\nu_{\delta}(\bm{t})=d\nu_{\delta}(t_1,\dots,t_N)$ is a probability measure on $\mathbb{R}^{N}$, called the effective $t$-field of the $H^{2|2}$-model pinned at $\delta$. This coincides with~\cite[Equation~(5)]{MR3729620}, where the effective $t$-field pinned at $i_{0}\in V=\{1,2,\dots, N,\delta\}$ is defined as
\begin{equation}\label{eq:STZ_mixing}
    d\nu_{i_0}(\bm{t})=\mathbf{1}_{\{t_{i_0}=0\}}e^{-\frac{1}{2}\sum_{i,j\in V}W_{ij}(\cosh(t_i-t_j)-1)}\sqrt{\mathcal{D}_W(\bm{t})}\prod_{i\in V, i\neq i_0}\frac{e^{-t_i}dt_i}{\sqrt{2\pi}}.
\end{equation}

This provides a setting in which one can apply a function $f:\mathbb{R}\to\mathbb{R}$ to the sum $\sum_{k}p_k(x_k+z_k)=\sum_{k}p_ke^{t_k}\geq 0$ without supersymmetry, since the latter acquires a probabilistic meaning in the horospherical representation.

\subsubsection{Example: rooted one-point graph}\label{rooted-1p-graph}
Consider the rooted one-point graph with $V=\{1,\delta\}$. We assume $p_1=1$ by scaling and write, with $W=W_{1\delta}$ and $z=z_1$, the variation
\begin{equation*}
    \frac{d}{dW}\int D\mu f(x+z)e^{-W(z-1)}=\int D\mu f(x+z)(1-z)e^{-W(z-1)}=-\int D\mu f(x+z)Q(\tilde{\lambda})e^{-W(z-1)}
\end{equation*}
where $\tilde{\lambda}=\frac{\lambda}{1+z}$ with $\lambda=x\eta-y\xi$, so that $Q(\tilde{\lambda})=z-1$. Integrating by parts in $Q$,
\begin{equation*}
    -\int D\mu f(x+z)Q(\tilde{\lambda})e^{-W(z-1)}=\int D\mu Q(f(x+z))\tilde{\lambda}e^{-W(z-1)}=\int D\mu f'(x+z)(\xi\tilde{\lambda})e^{-W(z-1)}.
\end{equation*}
Notice that $\xi\lambda=Q(\nu)$ where $\nu=-xy\xi-y^2\eta$. Therefore with $\tilde{\nu}=\frac{\nu}{1+z}$ we get
\begin{equation*}
    \int D\mu f'(x+z)Q(\tilde{\nu})e^{-W(z-1)}=-\int D\mu f''(x+z)\xi \tilde{\nu}e^{-W(z-1)}=\int D\mu f''(x+z) \frac{y^2\xi\eta}{1+z}e^{-W(z-1)}.
\end{equation*}
Now we change to the horospherical coordinates~\eqref{eq:HorosphericalCoordinates}. Because the integrand already contains the top-degree Grassmann element $\xi\eta$, any additional fermionic terms will yield a product containing $\xi^2$ or $\eta^2$, which vanishes. Therefore, we can evaluate $z$ strictly at its even ``scalar body'' $\text{bd}(z)=\cosh(t)+\frac{1}{2}s^2e^{t}>0$. Integrating out the fermions then yields
\begin{equation*}
    \int \frac{dtds}{2\pi}\partial_{\bar{\psi}}\partial_\psi\, e^{-t}f''(e^{t})s^2e^{4t}\bar{\psi}\psi\frac{e^{-W(z-1)}}{1+z}=-\int \frac{dtds}{2\pi} e^{-t}f''(e^{t})\frac{s^2e^{4t}}{1+\text{bd}(z)}e^{-W(\text{bd}(z)-1)}\leq 0
\end{equation*}
by the convexity of $f$ and $s^2\geq 0$. Hence, we have proved that
\begin{equation}
  \label{eq-mono-1p-h22-susy}
  \frac{d}{dW}\int D\mu f(x+z)e^{-W(z-1)} \le 0.
\end{equation}
This proves the $H^{2|2}$ monotonicity theorem with $N=1$.

\begin{rema}\label{rema:ReductionProperty}
  This inequality is a direct replacement for the discrete coupling argument of~\cite[Theorem~5]{MR4721024}. Combined with a known graph reduction property for the $H^{2|2}$ model~\cite[Proposition~1.2.1]{MR4721024} (see also~\cite[Lemma~5]{MR3904155} and~\cite{zbMATH07162897}), this already yields an alternative proof of the $H^{2|2}$ monotonicity theorem. See Appendix \ref{short-cut} for more details.
\end{rema}

\subsection{General \texorpdfstring{$H^{2|2}$}{H2|2} monotonicity theorem}\label{subse:general_h_2_2}
In the sequel, we use the new supersymmetric integration by parts method to derive a simple and self-contained proof of a stronger $H^{2|2}$ monotonicity theorem, with a Kahane--Slepian type joint convexity condition. This short section is self-contained and does not rely on any input from~\cite{MR4721024}. For readability, we recall the stronger $H^{2|2}$ monotonicity theorem already introduced in the introduction.

\begin{theo}[Generalized form of the $H^{2|2}$ monotonicity theorem]\label{th:better}
Suppose $F: \mathbb{R}^N \to \mathbb{R}$ is smooth and jointly convex (i.e., it has a positive semi-definite Hessian). Then the integral
\begin{equation*}
    J=\int F(e^{t_1},\dots,e^{t_N})d\nu_\delta(\bm{t})
\end{equation*}
is non-increasing in each of the edge weights $W_{ij}$ with $i,j\in\{1,\dots,N,\delta\}$. 
\end{theo}
The smoothness assumption can be dropped by standard approximation arguments. The integral could be non-finite, in which case it evaluates to $+\infty$: the non-increasing property means then that the integral stays infinity when the edge weights decrease. A slightly generalized formulation of Theorem~\ref{th:main} follows at once since for \emph{any real parameters} $p_1,\dots,p_N\in\mathbb{R}$ without any sign condition, the function
\begin{equation*}
    F(e^{t_1},\dots,e^{t_N})=f\left(\sum_{k=1}^{N}p_ke^{t_k}\right)
\end{equation*}
is jointly convex in the variables $e^{t_1},\dots,e^{t_N}$ when $f$ is convex.

\subsection{A useful lemma}
We now state a useful ``switching lemma'' to simplify many of the calculations in the sequel. In what follows, we denote by $\bm{x}+\bm{z}$ the vector $(x_1+z_1,\dots,x_N+z_N)$, and by $\partial_{k} F$ and $\partial_{k\ell} F$ the partial derivatives of $F$.

\begin{lemm}[Switching lemma]\label{lemm:Switching}
Consider a smooth function $F:\mathbb{R}^{N}\to\mathbb{R}$ with sufficient decay and a $Q$-invariant even function $A$. We assume furthermore that every monomial in the expansion of $A$ has equal degree of $\xi$ and $\eta$. For any $i,j,k\in\{1,\dots,N\}$, we have
\begin{equation*}
    \int d\mu\,F(\bm{x}+\bm{z}) x_k\xi_i\eta_j A= -\int d\mu F(\bm{x}+\bm{z}) x_ky_iy_j A.
\end{equation*}
A further $Q$-integration by parts shows that this quantity also equals
\begin{equation*}
    \sum_{\ell=1}^{N}\int d\mu\,\partial_\ell F(\bm{x}+\bm{z}) \xi_\ell\eta_ky_iy_j A.
\end{equation*}
\end{lemm}

\begin{proof}
Recall that the integral with respect to $d\mu$ of a fermionic form $\theta$ having unequal numbers of $\xi$ and $\eta$ components vanishes: $\int d\mu\, \theta = 0.$

Write $\eta_j=Q(y_j)$, and apply the $Q$-integration by parts formula with the odd variable $F(\bm{x}+\bm{z})x_k\xi_i$ and the $Q$-invariant function $A$:
\begin{equation*}
    \int d\mu\,F(\bm{x}+\bm{z}) x_k\xi_iQ(y_j) A= \int d\mu\,Q(F(\bm{x}+\bm{z}) x_k\xi_i)y_j A.
\end{equation*}
Applying the Leibniz rule yields $Q(F(\bm{x}+\bm{z}) x_k\xi_i) = Q(F(\bm{x}+\bm{z})x_k)\xi_i + F(\bm{x}+\bm{z})x_kQ(\xi_i)$. The first term vanishes upon integration because it contains two more $\xi$ components than $\eta$ components (using the equal fermionic degree assumption on $A$). Using $Q(\xi_i)=-y_i$, we obtain
\begin{equation*}
    -\int d\mu\,F(\bm{x}+\bm{z}) x_ky_iy_j A,
\end{equation*}
which proves the first identity.

The second identity follows similarly. Write $x_k=Q(\eta_k)$ and perform $Q$-integration by parts with the $Q$-invariant function $A$:
\begin{equation*}
    -\int d\mu\,F(\bm{x}+\bm{z}) Q(\eta_k)y_iy_j A= -\int d\mu\,\eta_k Q(F(\bm{x}+\bm{z})y_iy_j) A.
\end{equation*}
Expand the argument using the Leibniz rule: $\eta_k Q(F(\bm{x}+\bm{z})y_iy_j) = \eta_k Q(F(\bm{x}+\bm{z}))y_iy_j + \eta_k F(\bm{x}+\bm{z})Q(y_iy_j)$. The second term vanishes upon integration because it carries an excess of two $\eta$ components. We conclude that
\begin{equation*}
    -\int d\mu \,\eta_k Q(F(\bm{x}+\bm{z})y_iy_j) A= \sum_{\ell=1}^{N}\int d\mu \,\partial_\ell F(\bm{x}+\bm{z})\xi_\ell\eta_ky_iy_j A.
\end{equation*}
This completes the proof.
\end{proof}

We now prove Theorem~\ref{th:better}. We start by studying the variation with respect to a boundary edge weight, then the variation in the bulk.

\subsubsection{Boundary variations}\label{subse:boundary_variations}
We first prove the monotonicity with respect to a fixed boundary edge weight, say $W_{1\delta}$.

\begin{proof}[Proof of Theorem~\ref{th:better} in the boundary case]
Without loss of generality, consider the variation with respect to $W_{1\delta}$. With $\lambda_{1}=x_1\eta_1-y_1\xi_1$ as in the previous examples,
\begin{equation*}
\begin{split}
    \frac{\partial}{\partial W_{1\delta}}J&=\int D\mu\,F(\bm{x}+\bm{z})(1-z_1)e^{-S^{W}}\\
    &=\int D\mu\,F(\bm{x}+\bm{z})Q\left(\frac{-\lambda_1}{1+z_1}\right)e^{-S^{W}}\\
    &=\sum_{k=1}^{N}\int D\mu\, \partial_k F(\bm{x}+\bm{z})\xi_k\frac{\lambda_1}{1+z_1}e^{-S^{W}}.
\end{split}
\end{equation*}
At this stage one can replace $\lambda_1$ by $x_1\eta_1$, since the other part of $\lambda_1$ will create an excess of $\xi$ components and vanishes upon integration. Then we apply the switching lemma with the $Q$-invariant function $A = \left(\prod_{\ell=1}^{N} \frac{1}{z_{\ell}}\right) \frac{1}{1+z_{1}} e^{-S^{W}}$:
\begin{equation*}
    \frac{\partial}{\partial W_{1\delta}}J=\sum_{k=1}^{N}\int D\mu\, \partial_k F(\bm{x}+\bm{z})\xi_k\frac{x_1\eta_1}{1+z_1}e^{-S^{W}}=\sum_{k,\ell=1}^{N}\int D\mu\, \partial_{k\ell}F(\bm{x}+\bm{z})\xi_\ell y_k\frac{\eta_1y_1}{1+z_1}e^{-S^{W}}.
\end{equation*}
Using the expansion $\frac{1}{1+z_1}=\int_{0}^{\infty} dw e^{-2w} e^{-w(z_1-1)}$, the above integral can be rewritten as
\begin{equation}\label{eq:beforePassingToHoro}
    \int_{0}^{\infty}dw e^{-2w}\left[\sum_{k,\ell=1}^{N}\int D\mu\, \partial_{k\ell}F(\bm{x}+\bm{z})\xi_\ell y_k\eta_1y_1 e^{-S^{W,w}}\right],
\end{equation}
where the action $S^{W,w}$ is defined with a shifted weight $W_{1\delta}\mapsto W_{1\delta}+w$. It remains to show that for any fixed $w>0$, the following quantity is non-positive:
\begin{equation*}
    \sum_{k,\ell=1}^{N}\int D\mu\, \partial_{k\ell}F(\bm{x}+\bm{z})\xi_\ell y_k\eta_1y_1 e^{-S^{W,w}}.
\end{equation*}
In the horospherical coordinates~\eqref{eq:HorosphericalCoordinates} and $d\mu=\prod_{i=1}^{N}\frac{dt_ids_i\partial_{\bar{\psi}_i}\partial_{\psi_i}}{2\pi}$, the above display becomes
\begin{equation*}
    \sum_{k,\ell=1}^{N}\int d\mu\, e^{-\sum_{i}t_i}\partial_{k\ell}F(e^{\bm{t}})\bar{\psi}_\ell e^{t_\ell}s_ke^{t_k}\psi_{1}e^{t_1}s_1e^{t_1}e^{-S^{W,w}}.
\end{equation*}
Gaussian integration in $s,\bar{\psi},\psi$ using the covariance matrix $G^{W,w}_{t}=(\Delta^{W,w}_t)^{-1}$ recalled in~\eqref{eq:M-LaplaceMatrix} (with the $w$-shift in $W_{1\delta}$ as above) yields, with the effective measure~\eqref{eq:mixing_measure} but shifted by $w$,
\begin{equation*}
\begin{split}
    &\sum_{k,\ell=1}^{N}\int d\nu^{w}_{\delta}(\bm{t})\left[-e^{t_\ell+t_1}(G^{W,w}_t)_{\ell 1}\partial_{k\ell}F(e^{\bm{t}})e^{t_k+t_1}(G^{W,w}_t)_{k1}\right].
\end{split}
\end{equation*}
The minus sign comes from the fermionic expectation $\mathbb{E}[\bar{\psi}_i\psi_j]=-(G^{W,w}_t)_{ij}$. This is manifestly non-positive by the joint convexity of $F$, since the term inside the square brackets takes the form $-v^{T}(\mathrm{Hess}(F))v$.
\end{proof}
The above proof also yields an exact formula for the variation of $J$ with respect to a boundary edge weight. The mechanism behind the generalized $H^{2|2}$ monotonicity property becomes transparent: this variation is the average of expressions of the form $-v^{T}(\mathrm{Hess}(F))v$, so joint convexity is the correct condition to impose in view of this explicit calculation.

\begin{rema}
Notice that the positivity of the entries of the Green matrix $G$ is not used. Therefore, our proof does not rely on the $M$-matrix property of the $H^{2|2}$ model, which was crucial for probabilistic methods such as random path expansion.
\end{rema}

\subsubsection{Bulk variations}
We now prove the monotonicity with respect to the parameter $W_{ij}$ for a fixed pair $i,j\in\{1,\dots,N\}$: without loss of generality, suppose that $(i,j)=(1,2)$. The basic idea is to use a ``rerooting'' formula to shift the pinning condition from the vertex $\delta$ to the vertex $1$. After this rerooting, $W_{12}$ becomes a boundary edge weight and we can directly invoke the previous proof in the boundary case.

\begin{lemm}[Rerooting formula]\label{lemm:rerooting}
Let $\delta,b\in V$ and consider the measures $d\nu_\delta$ and $d\nu_b$ with different pinning points $\delta$ and $b$, defined in~\eqref{eq:mixing_measure}. If the variables $\bm{t}'$ and $\bm{t}$ are related by $t'_i=t_i-t_b$ for all $i\in V$, then
\begin{equation*}
    e^{t_b-t_\delta}d\nu_\delta(\bm{t})=d\nu_b(\bm{t}').
\end{equation*}
\end{lemm}
\begin{proof}
Recall from~\cite[Equation~(5)]{MR3729620} that
\begin{equation*}
    d\nu_{i_0}(\bm{t})=\mathbf{1}_{\{t_{i_0}=0\}}e^{-\frac{1}{2}\sum_{i,j}W_{ij}(\cosh(t_i-t_j)-1)}\sqrt{\mathcal{D}_W(\bm{t})}\prod_{i\in V, i\neq i_0}\frac{e^{-t_i}dt_i}{\sqrt{2\pi}}.
\end{equation*}
The proof follows by performing the change of variables $t'_i=t_i-t_b$, and noticing that the scaling of $\sqrt{\mathcal{D}_W(\bm{t})}$ defined in~\eqref{eq:det_matrixtree} exactly cancels that of the $\prod e^{-t_i}$ factor in the density.
\end{proof}

Recall that if $F(x_1,\dots,x_N)$ is jointly convex, its perspective function $P(x_1,\dots,x_N,T)=T\cdot F(\frac{x_1}{T},\dots,\frac{x_N}{T})$ is jointly convex in $(x_1,\dots,x_N,T)$ for $T>0$, see Appendix~\ref{subse:perspective_function_of_a_jointly_convex_function}.

\begin{proof}[Proof of Theorem~\ref{th:better} in the bulk case]
Consider the variation with respect to the bulk edge weight $W_{12}$ of the quantity
\begin{equation*}
    J=\int F(e^{\bm{t}})d\nu_\delta(\bm{t}).
\end{equation*}
We reroot the integral from $\delta$ to $1$: by Lemma~\ref{lemm:rerooting} we get
\begin{equation*}
    J=\int e^{t'_{\delta}}F(e^{\bm{t}'}\cdot e^{-t'_\delta})d\nu_1(\bm{t}'),
\end{equation*}
where the new pinning is $t'_1=0$ and the integration is over the variables $t'_{i}$ for $i\in\{2,\dots,N,\delta\}$. Setting $T=e^{t'_\delta}$, we see that the new function
\begin{equation*}
    G(e^{\bm{t}'})=e^{t'_{\delta}}F(e^{\bm{t}'}\cdot e^{-t'_\delta})
\end{equation*}
is jointly convex in the variables $e^{t'_2},\dots,e^{t'_N},e^{t'_{\delta}}$. Since $W_{12}$ becomes a boundary edge weight after the rerooting, the generalized monotonicity theorem in Section~\ref{subse:boundary_variations} applies, and the proof in the bulk case is complete.
\end{proof}

\begin{coro}[Monotone random edge weights]
Suppose that $\{W(\bm{a})\}_{\bm{a}}$ is a family of random edge weights such that when $\bm{a}\leq\bm{a'}$, there is an almost sure coupling with $W_e(\bm{a})\leq W_e(\bm{a'})$ for every edge $e$. Assuming that the integrals are finite for some smooth jointly convex function $F$, then
\begin{equation*}
    \bm{a}\mapsto \mathbb{E}\left[\int F(e^{\bm{t}})d\nu_\delta^{W(\bm{a})}(\bm{t})\right]
\end{equation*}
is coordinate-wise non-increasing.

In particular, this applies to $W_e(a)=w^0_e+T_e(a)$, where $T_e=(T_e(a))_{a\geq 0}$ are subordinators. For Gamma processes this recovers the random-weight comparison used in the Edge Reinforced Random Walk setting (see~\cite[Theorem~1]{MR3420510} and~\cite[Proof of Theorem~3]{MR4721024}).
\end{coro}
\begin{proof}
This follows from integrating the deterministic Theorem~\ref{th:better} pathwise under the almost surely increasing coupling.
\end{proof}

\appendix
\section{Replacing Poudevigne's coupling argument by supersymmetric integration by parts}\label{short-cut}
While Section~\ref{subse:general_h_2_2} gives a self-contained proof of the more general multivariate convex $H^{2|2}$ monotonicity theorem, we demonstrate here how the $1$-point inequality (Section~\ref{rooted-1p-graph}) serves as a direct substitute for Poudevigne's discrete coupling lemma~\cite[Theorem 5]{MR4721024} for the original weaker formulation in Theorem~\ref{th:main}. This result can be combined directly with the existing graph-reduction framework~\cite{MR3729620,MR3904155}.

Equation~\eqref{eq-mono-1p-h22-susy}, written in the $H^{2|2}$ horospherical coordinates~\eqref{eq:HorosphericalCoordinates}, is the statement that for any convex function $f:\mathbb{R}\to\mathbb{R}$,
\begin{equation}\label{eq:2point}
    \frac{d}{dW} \int_{\mathbb{R}} f(e^{t}) e^{-W(\cosh(t)-1)} \sqrt{W e^{t}} \frac{e^{-t}dt}{\sqrt{2\pi}} \le 0.
\end{equation}
We introduce the random Schrödinger representation to use the graph reduction property of the $H^{2|2}$ model. Define the random Schrödinger operator $H_\beta$ on a finite graph $V$ with symmetric positive edge weights $W_{ij}>0$. Set $H_\beta(i,i) = 2\beta_i$ and $H_\beta(i,j) = -W_{ij}$ for $i \neq j$, where the $\beta_i$ are real-valued (retaining the notation of~\cite{MR3729620} despite our convention for Greek letters). \cite[Definition~1]{MR3729620} introduced a probability measure $q_V^{W,1}(d\beta)$; as the off-diagonal entries of $H_{\beta}$ are deterministic, $q_V^{W,1}(d\beta)$ is also the law of the random matrix $H_{\beta}$.

We apply this setup to the two-vertex graph \(V=\{1,\delta\}\) with edge weight \(W = W_{1\delta}\). By~\cite[Theorem~3]{MR3729620}, we can rewrite this horospherical integral~\eqref{eq:2point} as an expectation over a real random vector \(\beta = (\beta_1, \beta_\delta)\). Specifically, following~\cite[Definition~1]{MR3729620}, let
\begin{equation*}
    q^{W,1}_{\{1,\delta\}}(d\beta)=\mathbf{1}_{\{H_\beta>0\}}\frac{2}{\pi}e^{W}\frac{e^{-\beta_1-\beta_\delta}}{\sqrt{\det H_\beta}}d\beta_1 d\beta_\delta.
\end{equation*}
Let $G=H_\beta^{-1}$ denote the Green function. For any constants $p_{1},p_{\delta}\in \mathbb{R}$ and any convex function $g:\mathbb{R}\to\mathbb{R}$, define the convex function $f$ by $f(y)=g(p_{\delta}+p_{1}y)$. In~\cite[Theorem~3]{MR3729620}, the distributional identity $e^{t} = \frac{G(1,\delta)}{G(\delta,\delta)}$ was established, which yields an equivalent reformulation of~\eqref{eq:2point}:
\begin{equation}\label{eq:2pt_monotonicity}
    \frac{d}{dW} \int g\left(p_{\delta}+p_{1} \frac{G(1,\delta)}{G(\delta,\delta)}\right) q^{W,1}_{\{1,\delta\}}(d\beta) \le 0.
\end{equation}

On a general graph with $V=\{1,2,\dots,N,\delta\}$, let $q_{V}^{W,1}(d\beta)$ denote the distribution for the $(N+1) \times (N+1)$ random Schrödinger matrix $H_\beta$. Let $G = H_{\beta}^{-1}$ be its inverse. The explicit density of $H_{\beta}$ was given in~\cite[Definition~1]{MR3729620}, but we only need the following graph reduction property~\cite[Lemma~5]{MR3904155}: 
\begin{prop}\label{prop:reduction}
Partition the vertex set $V$ into two subsets: the bulk $V_1 = \{1, \dots, N-1\}$ and the boundary $V_2 = \{N,\delta\}$.\footnote{The boundary was $\{1,\delta\}$ in Section~\ref{subse:boundary_variations}; we choose $\{N,\delta\}$ here to cleanly perform the block decomposition.} This partition induces a block decomposition of the symmetric matrices $H_\beta$ and $G$:
\begin{equation*}
    H_\beta = \begin{pmatrix} H_\beta^{V_1, V_1} & H_\beta^{V_1, V_2} \\ H_\beta^{V_2, V_1} & H_\beta^{V_2, V_2} \end{pmatrix}, \quad G = \begin{pmatrix} G^{V_1, V_1} & G^{V_1, V_2} \\  G^{V_2, V_1} & G^{V_2, V_2} \end{pmatrix}.
\end{equation*}
Then conditionally on \((\beta_{i})_{i\in V_{1}}\), \((G^{V_2,V_2})^{-1}\) is distributed according to \(q_{\{N,\delta\}}^{\widetilde{W}_{N\delta},1}\), where
\[\widetilde{W}_{N\delta} = W_{N\delta} + \left[H_{\beta}^{V_{2},V_{1}} (H_\beta^{V_1, V_1})^{-1} H_{\beta}^{V_{1},V_{2}}\right]_{N\delta}.\]
In particular, \(\frac{d \widetilde{W}_{N\delta}}{dW_{N\delta}} =1\).
\end{prop}  

Observe that $G(\delta,\delta)=G^{V_2, V_2}(\delta,\delta)$ and $G(N,\delta)=G^{V_2, V_2}(N,\delta)$. Let $p_1, \dots, p_N, p_{\delta}$ be real numbers. \cite[Theorem~3]{MR3729620} provides the identity in law
\begin{equation}\label{eq:ratio}
    \sum_{j\in V}p_je^{t_j}=\frac{\sum_{j\in V} p_j G(j,\delta)}{G(\delta,\delta)}.
\end{equation}
Decompose the numerator sum over the subsets $V_1$ and $V_2$:
\begin{equation*}
    \sum_{j\in V} p_j G(j,\delta) = \sum_{j \in V_1} p_j G(j,\delta) + p_{N} G^{V_2, V_2}(N,\delta)+p_{\delta}G^{V_2, V_2}(\delta,\delta).
\end{equation*}
Invert the block matrix using the Schur complement formula. This yields $G^{V_1, V_2} = M G^{V_2, V_2}$ with $M = -(H_\beta^{V_1, V_1})^{-1} H_\beta^{V_1, V_2}$. The columns of $M$ are indexed by the boundary vertices $\{N,\delta\}$. Notice that $G(j,\delta)=[M G^{V_2, V_2}]_{j\delta}$ and we have
\begin{equation*}
    \sum_{j \in V_1} p_j G(j,\delta)=\sum_{j \in V_1} p_j \left(M_{jN} G^{V_2, V_2}(N,\delta) + M_{j\delta} G^{V_2, V_2}(\delta,\delta)\right).
\end{equation*}
Define \(a_N = \sum_{j \in V_1} p_j M_{jN}\) and \(a_\delta = \sum_{j \in V_1} p_j M_{j\delta}\). Grouping the terms by \(G^{V_2,V_2}\) components,
\begin{equation*}
    \sum_{j\in V} p_j G(j,\delta)=(a_N + p_{N}) G^{V_2,V_2}(N,\delta) + (a_\delta + p_\delta) G^{V_2,V_2}(\delta,\delta).
\end{equation*}
We now define the shifted parameters \(\tilde{p}_N = a_N + p_{N}\) and \(\tilde{p}_\delta = a_\delta + p_\delta\). The original ratio~\eqref{eq:ratio} reduces to a $2$-point boundary form:
\begin{equation*}
    \frac{\sum_{j\in V} p_j G(j,\delta)}{G(\delta,\delta)} = \tilde{p}_\delta + \tilde{p}_N \frac{G^{V_2,V_2}(N,\delta)}{G^{V_2,V_2}(\delta,\delta)}.
\end{equation*}
This algebraic identity constitutes the core of~\cite[Lemma 2.2.2]{MR4721024}. Moreover, \(a_{N},a_{\delta}\) depend only on \(\beta_{1} ,\ldots , \beta_{N-1}\), since \(M\) is entirely determined by \(H_{\beta}^{V_1,V_1}\) and \(H_{\beta}^{V_1,V_2}\).

Finally, Proposition~\ref{prop:reduction} yields that \((G^{V_2,V_2})^{-1}\) follows the law $q_{\{N,\delta\}}^{\widetilde{W}_{N\delta},1}$ associated with the effective edge weight $\widetilde{W}_{N\delta}$ satisfying $\frac{d \widetilde{W}_{N\delta}}{d W_{N\delta}} = 1$. Apply the $2$-point inequality~\eqref{eq:2pt_monotonicity},
\begin{equation*}
    \frac{d}{dW_{N\delta}} \mathbb{E}_{q_{V}^{W,1}}\left[ g\left(\tilde{p}_\delta + \tilde{p}_N \frac{G^{V_2, V_2}(N,\delta)}{G^{V_2, V_2}(\delta,\delta)}\right) \;\middle|\; \beta_{1},\ldots,\beta_{N-1} \right]  \leq 0.
\end{equation*}
Taking the expectation of this conditional inequality over the bulk variables \(\beta_1, \dots, \beta_{N-1}\) (which do not depend on the edge weight $W_{N\delta}$) yields the monotonicity theorem of the \(H^{2|2}\) model on the graph $V=\{1,\dots,N,\delta\}$.

\section{Perspective function of a jointly convex function}\label{subse:perspective_function_of_a_jointly_convex_function}
Consider a jointly convex function $F:\mathbb{R}^{N}\to\mathbb{R}$. This is equivalent to saying that
\begin{equation*}
    \forall \bm{x},\bm{y}\in\mathbb{R}^{N},\forall \lambda\in(0,1),\quad F(\lambda\bm{x}+(1-\lambda)\bm{y})\leq \lambda F(\bm{x})+(1-\lambda)F(\bm{y}).
\end{equation*}
Define its perspective function $P(x_1,\dots,x_N,T)=T\cdot F(\frac{x_1}{T},\dots,\frac{x_N}{T})$ on $\mathbb{R}^{N}\times\mathbb{R}_{>0}$. To show that $P$ is jointly convex, for all $\bm{x},\bm{y}\in\mathbb{R}^{N}$, $T_x,T_y>0$ and $\lambda\in(0,1)$, define
\begin{equation*}
    \theta=\frac{\lambda T_x}{\lambda T_x+(1-\lambda)T_y}\in(0,1)
\end{equation*}
and check that
\begin{equation*}
\begin{split}
    P\left(\lambda(\bm{x},T_x)+(1-\lambda)(\bm{y},T_y)\right)&=(\lambda T_x+(1-\lambda)T_y) F\left(\frac{\lambda\bm{x}+(1-\lambda)\bm{y}}{\lambda T_x+(1-\lambda)T_y}\right)\\
    &\leq(\lambda T_x+(1-\lambda)T_y)\left(\theta F\left(\frac{\bm{x}}{T_x}\right)+(1-\theta)F\left(\frac{\bm{y}}{T_y}\right)\right)\\
    &=\lambda T_x F\left(\frac{\bm{x}}{T_x}\right)+(1-\lambda) T_y F\left(\frac{\bm{y}}{T_y}\right)=\lambda P(\bm{x},T_x)+(1-\lambda)P(\bm{y},T_y).
\end{split}
\end{equation*}

\bigskip

\subsection*{Conflict of interest}
The authors have no relevant financial or non-financial interests to disclose.

\subsection*{Data availability}
Data sharing is not applicable to this article as no new data were created or analyzed in this work.

\bibliographystyle{alpha}

\end{document}